\documentclass[10pt,a4paper,oneside]{article}
\usepackage{graphicx}
\usepackage{epstopdf}
\usepackage[noblocks]{authblk}
\usepackage{amssymb}
\usepackage{amsmath}
\usepackage{amsfonts}
\usepackage{mathtools}
\usepackage{amsfonts}
\usepackage{color}
\usepackage{listings}
\usepackage{fancyvrb}
\usepackage{float}
\usepackage{framed}
\usepackage{multirow}
\usepackage{hyperref}
\usepackage{url}  % This makes \url work
\usepackage{multicol}
\usepackage{subcaption}
\usepackage{algorithm}
\usepackage{algpseudocode}% http://ctan.org/pkg/algorithmicx
\usepackage{pifont}
\usepackage{algorithm}
\usepackage{colortbl}
\usepackage{framed}
\usepackage{listings}
\usepackage{bm}
\usepackage{eurosym}
\usepackage{amssymb}% http://ctan.org/pkg/amssymb
\usepackage{suffix}
\usepackage{tikz}
\usetikzlibrary{matrix}

\usepackage[a4paper,left=3.5cm,right=3.5cm,top=3.5cm,bottom=3.5cm]{geometry}

\newtheorem{theorem}{Theorem}[section]
\newtheorem{proposition}[theorem]{Proposition}

\newenvironment{proof}{\noindent{\bf Proof:}}{\hfill$\square$}

\DeclarePairedDelimiterX\MeijerM[3]{\lparen}{\rparen}%
{\begin{smallmatrix}#1 \\ #2\end{smallmatrix}\delimsize\vert\,#3}

\newcommand\MeijerG[8][]{%
  G^{\,#2,#3}_{#4,#5}\MeijerM[#1]{#6}{#7}{#8}}

\WithSuffix\newcommand\MeijerG*[7]{%
  G^{\,#1,#2}_{#3,#4}\MeijerM*{#5}{#6}{#7}}

\DeclareMathOperator{\csch}{csch}
\DeclareMathOperator{\sech}{sech}

\title{Numerical methods and arbitrary-precision computation of the Lerch transcendent}
\author{Guillermo Navas-Palencia} 
\date{}
\begin{document}
\maketitle

\begin{abstract}
We examine the use of the Euler-Maclaurin formula and new derived uniform asymptotic expansions for the numerical evaluation of the Lerch transcendent $\Phi(z, s, a)$ for $z, s, a \in \mathbb{C}$ to arbitrary precision. A detailed analysis of these expansions is accompanied by rigorous error bounds. A complete scheme of computation for large and small values of the parameters and argument is described along with algorithmic details to achieve high performance. The described algorithm has been extensively tested in different regimes of the parameters and compared with current state-of-the-art codes.
An open source implementation of $\Phi(z, s, a)$ based on the algorithms described in this paper is available.
\end{abstract}

\section{Introduction}

The Lerch transcendent, also called Hurwitz-Lerch zeta function, which is named after the Czech mathematician Mathias Lerch (1860 - 1922) is defined by means of the Dirichlet series \cite{apostol1951}
\begin{equation}\label{l_series}
\Phi(z,s,a) = \sum_{k=0}^{\infty} \frac{z^k}{(k+a)^s},
\end{equation}
where $\Phi(z,s,a)$ is absolutely convergent for $|z| \le 1, \; a \notin \mathbb{Z}_0^{-}$ or  $\Re(s) > 1, \; |z| = 1$ and is defined elsewhere by analytic continuation. The Lerch transcendent serves as a unified framework for the study of various particular cases of special functions in number theory  such as polygamma functions, polylogarithms, Dirichlet $L$-functions and certain number-theoretical constants. 
The Lerch transcendent is related to Lipschitz-Lerch zeta function by the functional equation
\begin{equation*}
\mathcal{R}(a, x, s) = \Phi(e^{2\pi i x}, s, a).
\end{equation*} 
This function was introduced and investigated by Lerch \cite{Lerch1887} and Lipschitz \cite{Lipschitz1857}, where the latter studied general Euler integrals including the Lerch zeta function. Subsequently, many authors have studied properties of these functions. Among the recent investigations on the analytic properties of Lerch zeta function, we remark the work conducted by Laurin\v{c}ikas and Garunk\v{s}tis in \cite{Laurincikas2002}. 

The Lerch transcendent and their special cases are ubiquitous in theoretical physics. They play a relevant role in particle physics, thermodynamics and statistical mechanics, being present, for instance, in Bose-Einstein condensation distribution \cite{Griffin1997} and integrals of the Fermi-Dirac distribution. They also occur in quantum field theory, in particular in quantum electrodynamic bound state calculations \cite{Jentschura1997}. Regarding mathematical applications, the Lerch zeta function can be used to evaluate Dirichlet $L$-series of the form
\begin{equation*}
L(s, \chi) = \sum_{k=1}^{\infty} \frac{\chi(k)}{k^s},
\end{equation*}
where $\chi : (\mathbb{Z}/q\mathbb{Z})^* \rightarrow \mathbb{C}$ is a Dirichlet character  and $q$ a natural number, thus the above summation is also expressible as a combination of Hurwitz zeta functions $\zeta(s, a)$ or polygamma functions for $s \in \mathbb{N}$
\begin{equation*}
L(s, \chi)  = \sum_{r=1}^q \chi(r) \sum_{n=0}^{\infty} \frac{1}{(r+nq)^s} = q^{-s} \sum_{r= 1}^q \chi(r) \zeta(s, r/q).
\end{equation*}

The Lerch transcendent occasionally occurs in statistics, for instance, it provides an analytic expression for the central moments of the geometric distribution.

Over the last two decades several authors have devised new series representations to extend the regime of computation of the Lerch transcendent. Complete asymptotic expansions including error bounds of $\Phi(z, s, a)$ for large $a$ and large $z$ are derived in \cite{Ferreira2004}. More recently, an \textit{exponentially-improved} expansion for the Lerch zeta function in large $a$ asymptotic was examined in \cite{Paris2016}. A remarkable and extensive review of properties, identities and numerical methods for the computation of the Lerch transcendent and their special cases was carried out by R. Crandall in  \cite{Crandall2012}. In addition, we mention two important convergent series: the Hasse's convergent series expansion in \cite{Guillera2008} given by
\begin{equation*}
(1-z) \Phi(z,s,a) = \sum_{n=0}^{\infty} \left(\frac{-z}{1-z}\right)^n \sum_{k=0}^n (-1)^k \binom{n}{k} (a+k)^{-s},
\end{equation*}
which holds for $s,z \in \mathbb{C}$ with $\Re(z) < 1/2$ and Erd\'{e}lyi-series representation \cite{Erdelyi1953}
\begin{equation*}
z^a \Phi(z,s,a) = \sum_{k=0}^{\infty} \zeta(s-k, a) \frac{\log^k(z)}{k!} + \Gamma(1-s) (-\log (z))^{s-1},
\end{equation*}
where $s$ is  not a positive integer, and for parameter $a\in (0,1]$, $|\log(z)|< 2\pi$, the series representation is linearly convergent.

Finally, the Hermite-type integral representation is given by
\begin{align}\label{int2}
\Phi(z, s, a) &= \frac{1}{2a^s} + \frac{(-\log(z))^{s-1}}{z^a} \Gamma(1-s, -a\log(z))\nonumber\\
&+ 2 \int_0^{\infty} \frac{\sin(s \arctan(t/a) - t \log(z))}{(a^2+t^2)^{s/2} (e^{2\pi t} - 1)} \mathop{dt}, \quad \Re(a) > 0.
\end{align}

In this paper, we derive complete new uniform asymptotic expansions of $\Phi(z, s, a)$ for large order of the parameters $a$, $s$ and argument $z$, with special emphasis on the less investigated case $\Re(z) \gg 0$. The starting point for our asymptotic expansions is the integral representation in (\ref{int2}). Additionally, a careful treatment of the Euler-Maclaurin formula is considered along with the calculation of a rigorous error bound. A significant effort have been made to develop uniform asymptotic expansions with tractable coefficients in terms of known entities and amenable to arbitrary-precision computations. An extensive discussion on algorithmic aspects for their successful implementation is also provided.

The outline of the paper is the following: in Section 2 we study the main numerical methods considered for the numerical evaluation of $\Phi(z,s,a)$, including error bounds. Then, in Section 3, we discuss in detail implementation aspects, several heuristics and performance issues. We also devise an effective algorithm that permits computation to arbitrary-precision in an extensive region of the function's domain. In Section 4, we provide numerical calculations and compare the present implementation with open source and commercial state-of-the-art libraries. Finally, in Section 5, we discuss possible enhancements and present our conclusions.

\section{Numerical methods}

\subsection{Euler-Maclaurin formula}

We briefly summarized the Euler-Maclaurin formula and refer to \cite{Cohen2007} for a formal proof. We closely follow the expository style in \cite{Johansson2015}. Let us suppose that $f$ is an analytic function on a closed domain $[N, U]$ where $N, U \in \mathbb{Z}$, and let $M$ be a positive integer. Let $B_n$ denote the $n$-th Bernoulli number and $\tilde{B}_{2M}(t) = B_n(t-\lfloor t\rfloor)$ denote the $n$-th periodic Bernoulli polynomials. The Euler-Maclaurin summation formula states that 
\begin{equation}\label{sum_series_em}
\sum_{k=N}^U f(k) = I + T + R
\end{equation}
where
\begin{align}
I &= \int_N^U f(t) \mathop{dt}\\
T &= \frac{1}{2} (f(N) + f(U)) + \sum_{k=1}^M \frac{B_{2k}}{(2k)!} \left(f^{(2k-1)}(N) - f^{(2k-1)}(U)\right)\\
R &= -\int_N^U \frac{\tilde{B}_{2M}(t)}{(2M)!} f^{(2M)}(t) \mathop{dt}.\label{sum_series_em_remainder} 
\end{align}
If $f$ decreases sufficiently rapid, letting $U \to \infty$ the above equations remain valid.

\begin{proposition}
The Euler-Maclaurin summation formula for the Lerch transcendent is given by
\begin{equation}
\Phi(z,s,a) = S + I + T + R,
\end{equation}
where
\begin{align}
S &= \sum_{k=0}^{N-1} \frac{z^k}{(k+a)^s},\\
I &= \frac{(-\log(z))^{s-1}}{z^a} \Gamma(1-s, -(a+N)\log(z)), \label{em_int}\\
T &= \frac{z^N}{(a+N)^s} \left( \frac{1}{2} + \sum_{k=1}^M \frac{B_{2k}}{(2k)!} \frac{U(-2k+1, -2k+2-s, -(a+N)\log(z))}{(a+N)^{2k-1}} \right),\label{em_tail}\\
R &= -\int_N^{\infty} \frac{\tilde{B}_{2M}(t)}{(2M)!} \frac{z^t}{(a+t)^{s+2M}} U(-2M, -2M+1 -s, -(a+t)\log(z)) \mathop{dt} \label{remainder}.
\end{align}
\end{proposition}

\begin{proof}
Let us first consider the Hermite-type integral in (\ref{int2})
\begin{equation}
I :=  \int_0^{\infty} \frac{\sin(s \arctan(t/a) - t \log(z))}{(a^2+t^2)^{s/2} (e^{2\pi t} - 1)} \mathop{dt}.
\end{equation}
For $z, s, a \in \mathbb{R}$, $z > 0$ and $a > 0$, the above integral can be written in the form
\begin{equation}\label{int3}
I = \frac{1}{a^s} \Im \left( \int_0^{\infty} \frac{z^{-it}}{(1-it/a)^s}\frac{\mathop{dt}}{e^{2\pi t} - 1}\right).
\end{equation}
The domain delimited by previous constraints shall be extended by analytic continuation. Now we express the integrand in (\ref{int3}) in terms of the confluent hypergeometric function $U(a,b,z)$ which yields
\begin{equation}\label{int4}
I = \frac{1}{{a^s}} \Im \left((-a\log(z))^s \int_0^{\infty} \frac{e^{-i\log(t)} U(s, s+1, (it -a)\log(z))}{e^{2\pi t}-1} \mathop{dt}\right),
\end{equation}
By applying the addition theorem for $U(a,b,z)$ \cite[\S 13.13]{NIST:DLMF} given by
\begin{equation}
U(a,b,x+y) = e^y \sum_{n=0}^{\infty} \frac{(-y)^n}{n!}U(a, b+n, x), \quad |y| < |x|.
\end{equation}
the integrand can be written as a summation defined by
\begin{equation}\label{u_addition_theorem}
e^{-i\log(t)} U(s, s+1, (it -a)\log(z)) = \sum_{k=0}^{\infty} \frac{(-i t \log(z))^k U(s, s+k+1, -a \log(z))}{k!}.
\end{equation}
Substituting (\ref{u_addition_theorem}) into (\ref{int4}) and formally interchanging summation and integration we obtain
\begin{equation*}
I = \frac{1}{{a^s}}\Im \left((-a\log(z))^s \sum_{k=0}^{\infty} \frac{(-i \log(z))^k U(s, s+k+1, -a \log(z))}{k!} \int_0^{\infty} \frac{t^k}{e^{2\pi t} - 1} \mathop{dt} \right),
\end{equation*}
where the integral can be directly evaluated in closed form by
\begin{equation*}
\int_0^{\infty}\frac{t^k}{e^{2\pi t} - 1} \mathop{dt} = \frac{k!}{(2\pi)^{k+1}} \zeta(k+1).
\end{equation*}

We use Kummer's transformation $U(s, s+k+1, -a \log(z)) (-a\log(z))^{k+s} = U(-k,1-k-s,-a\log(z))$ to rewrite $I$ in the form
\begin{equation}\label{summ_im}
I = \frac{1}{a^s} \Im \left(\sum_{k=1}^{\infty} \frac{i^k}{a^k} \frac{U(-k,1-k-s,-a\log(z))}{(2\pi)^{k+1}} \zeta(k+1)\right).
\end{equation}

Note that the same summation formula can be derived by expanding $f(t) = z^{-it}(1-it/a)^{-s}$, which gives
\begin{equation}\label{expand_f}
f(t) = \sum_{k=0}^{\infty} \left(\frac{-i t}{a}\right)^k  \frac{1}{k!} \sum_{j=0}^{\infty} \binom{k}{j} (-1)^j (a \log(z))^{k-j} \sum_{m=0}^j (-1)^{j-m} s(j, m) s^m,
\end{equation}
where $s(j,m)$ are Stirling numbers of the first kind. The inner summation in (\ref{expand_f}) is expressible in terms of rising factorial or Pochhammer's symbol $(s)_j$ using the well-known identities
\begin{equation}
\sum_{m=0}^j (-1)^{j-m} s(j, m) s^m = (-1)^j (-s)^{(j)} = (s)_j
\end{equation}
and
\begin{equation}
\sum_{j=0}^{\infty} \binom{k}{j} (a \log(z))^{k-j} (-s)^{(j)} = (-1)^k U(-k, 1-k-s, -a \log(z)).
\end{equation}

Finally, taking the imaginary part of (\ref{summ_im}) yields
\begin{align}\label{int_u_expan_final_form}
I &= \frac{1}{a^s} \sum_{k=0}^{\infty} \frac{(-1)^k}{a^{2k+1}} \frac{U(-2k-1,-2k-s,-a\log(z))}{(2\pi)^{2k+2}} \zeta(2k+2)\nonumber\\
&= \frac{1}{2a^s}\sum_{k=1}^{\infty} \frac{B_{2k}}{(2k)!} \frac{U(-2k+1,-2k+2-s,-a\log(z))}{a^{2k-1}},
\end{align}
where the relationship between Bernoulli numbers $B_{2k}$ and the Riemann zeta function is applied.
\end{proof}

Note that the expansion is convergent for $|\log(z)| < 2\pi$. This can be observed by taking the asymptotic estimate of the $k$-th term in (\ref{int_u_expan_final_form})
\begin{align}
|t_k| &= \left|\frac{B_{2k}}{(2k)!} \frac{U(-2k+1, -2k+2-s, -(a+N)\log(z))}{(a+N)^{2k-1}}\right|\nonumber\\ 
&\sim \frac{2}{(2\pi)^{2k}} |\log(z)|^{2k-1},\label{tk_asymp}
\end{align}
as $N, M \to \infty$, where we consider the usual asymptotic estimates for $U(a, b, z) \sim z^{-a}$ as $|z| \to \infty$ and $|B_{2k}| / (2k)!$ using the fact that $\zeta(2k) \sim 1$ as $k\to \infty$. To assess the domain of convergence for  $z$ we use the ratio test (d'Alembert ratio test)
\begin{equation*}
\lim_{k\to \infty} \left| \frac{t_{k+1}}{t_k} \right| \sim \frac{|\log(z)^2|}{4\pi^2} < 1 \Longleftrightarrow |\log(z)| < 2\pi.
\end{equation*}

Finally, taking $M$ such that $\Re(s) + 2M - 1  > 0$, the remainder term (\ref{remainder}) in the Euler-Maclaurin summation formula is well defined, giving its analytic continuation to $s \in \mathbb{C}\setminus\{1\}$.

\begin{theorem}\label{theorem_em} Given $a, s, z \in \mathbb{C}$ with $|\log(z)| < 2\pi$ and $N, M \in \mathbb{N}$ such that $\Re(a) + N > 0$ and $\Re(s) + 2M > 1$, the error term (\ref{remainder}) in the Euler-Maclaurin summation formula satisfies 
\begin{equation}\label{remainder_bound}
|R| \le \frac{4}{(2\pi)^{2M}} \left| C \sum_{k=0}^{2M} \binom{2M}{k} Q(k+1-2M-s, W)\frac{(-\log(z))^{-k-1+2M+s}} {\log^{-k}(z)}  \right|,
\end{equation}
where $C = \Gamma(1-s)/z^a$, $W = -(a+N)\log(z)$ and $Q(a,z) = \Gamma(a,z)/\Gamma(a)$ is the regularized incomplete Gamma function.
\end{theorem}

\begin{proof}
We have
\begin{align*}
|R| & = \left| \int_N^{\infty} \frac{\tilde{B}_{2M}(t)}{(2M)!} \frac{z^t}{(a+t)^{s+2M}} U(-2M, -2M+1 -s, -(a+t)\log(z)) \mathop{dt} \right|\\
& \le   \frac{\left|\tilde{B}_{2M}(t)\right|}{(2M)!}  \left| \int_N^{\infty} \frac{z^t}{(a+t)^{s+2M}} U(-2M, -2M+1 -s, -(a+t)\log(z)) \mathop{dt}\right|\\
&  \le \frac{4}{(2\pi)^{2M}} \left| \sum_{k=0}^{2M} \binom{2M}{k} B_k^M(s) \int_N^{\infty} \frac{z^t}{(a+t)^{s+2M}} ((a+t)\log(z))^k \mathop{dt}  \right|
\end{align*}
with $B_k^M(s) = (k+1-2M-s)_{2M-k}$. We apply the usual upper bound for $|\tilde{B}_n(t)| < 4n!/(2\pi)^n$ and formally interchange integration and the expansion of $U(-2M, -2M+1-s, -(a+N) \log(z))$ given by
\begin{equation*}
U(-2M, -2M+1-s, -(a+N) \log(z)) = \sum_{k=0}^{2M} \binom{2M}{k} B_k^M(s) ((a+t)\log(z))^k.
\end{equation*}

The integral above can be expressed in terms of the incomplete Gamma function $\Gamma(a,z)$ as follows (similar to (\ref{em_int}))
\begin{align}\label{bound_incgamma}
\int_N^{\infty} \frac{z^t}{(a+t)^{s+2M}} ((a+t)\log(z))^k \mathop{dt} & = \frac{\log^k(z)}{z^a} (-\log(z))^{-k-1+2M+s} \nonumber\\
& \times \Gamma(k+1-2M-s, -(a+N)\log(z))
\end{align}

Thus, we have
\begin{align*}
& \sum_{k=0}^{2M} \binom{2M}{k} B_k^M(s) \int_N^{\infty} \frac{z^t}{(a+t)^{s+2M}} ((a+t)\log(z))^k \mathop{dt}\\
& =  \frac{\Gamma(1-s)}{z^a} \sum_{k=0}^{2M} \binom{2M}{k} Q(k+1-2M-s, -(a+N)\log(z))\frac{(-\log(z))^{-k-1+2M+s}} {\log^{-k}(z)},
\end{align*}
where we use $B_k^M(s) = \Gamma(1-s)/\Gamma(k+1-2M-s)$. Note that for $s \in \mathbb{N}$ we take equation (\ref{bound_incgamma}) to avoid the pole.
\end{proof}

The bound given in Theorem \ref{theorem_em} give us a notably tight approximation of remainder (\ref{remainder}). However, for large $M$ the direct evaluation of the terminating series in (\ref{remainder_bound}) might be substantially expensive, being a not negligible part of the total computation time, therefore approximations for large order will be considered in Section \ref{section_algorithms}.

\subsection{Uniform asymptotic expansion for $\Phi(z, s, a)$}

A suitable Laplace-type integral representation of $\Phi(z, s, a)$ amenable to derive multiple asymptotic expansions \cite{Crandall2012}, is given by
\begin{equation}
\Phi(z, s, a) = \frac{1}{\Gamma(s)}\int_0^{\infty} \frac{t^{s-1}e^{-at}}{1-ze^{-t}} \mathop{dt}, \quad \Re(s) > 1, \; \Re(a) > 0, \; z \notin [1, \infty),
\end{equation}
which serves to define the analytic continuation of the Lerch-series to $z\in \mathbb{C}\setminus [1, \infty)$. This integral has been chosen as starting point to derive asymptotic expansions for either large or small $a$ (assuming that $s$ and $z$ are fixed) or for large $z$ in \cite{Ferreira2004}, and for a Bernoulli-series representation as in \cite{Crandall2012}. 
The aim of this subsection is to extend the domain of computation of the Poincar{\'e} type asymptotic expansion for large $a$ defined in \cite{Ferreira2004} by constructing a uniform asymptotic expansion for large $a$, $s$ and $z$.
 
We proceed to construct that expansion by using the vanishing saddle point method described in \cite{Temme2015}. This method is fundamentally a modification of Laplace's method applicable to integrals of the form
\begin{equation}
F_{\lambda}(z) = \frac{1}{\Gamma(\lambda)} \int_0^{\infty} t^{\lambda - 1} e^{-zt} f(t) \mathop{dt},\label{uae_int}
\end{equation}
with $\Re(\lambda) > 0$ and $z$ large, in which $\lambda$ might also be large. The resulting expansion is given by
\begin{equation*}
F_{\lambda}(z) \sim \sum_{k=0}^{\infty} \frac{a_k(\mu)P_k(\lambda)}{z^{k+\lambda}},
\end{equation*} 
where $a_k(\mu)$ are the coefficients of the expansion of $f(t)$ at the saddle point $\mu=\lambda/z$ and coefficients $P_k(\lambda)$ are expressible in terms of generalized Laguerre polynomials defined by
\begin{equation}\label{tricomi_carlitz}
P_k(\lambda)=k! L_k^{-k-\lambda} (-\lambda).
\end{equation}

At this point, we briefly recall the definition of the Eulerian polynomial and its connection with the polylogarithm function before stating the next proposition.

The Eulerian polynomial is defined as
\begin{equation}
A_k(z) = \sum_{j=0}^{k-1}\genfrac<>{0pt}{}{k}{j} z^j,
\end{equation}
where $\genfrac<>{0pt}{}{k}{j}$ are the Eulerian numbers \cite{Graham1994}. The Eulerian polynomials satisfy the recurrence equation
\begin{align}\label{eulerian_recursion}
A_0(z) &= 1, & A_k(z) &= \sum_{j=0}^{k-1}\binom{k}{j}A_j(z) (z-1)^{k-1-j}, \quad k \ge 1.
\end{align}

The Eulerian polynomial and polylogarithm are related by the functional equation
\begin{equation}
A_k(z) = \frac{(1-z)^{k+1}}{z} \text{Li}_{-k}(z) = \frac{(1-z)^{k+1}}{z} \sum_{j=1}^{\infty} j^k z^j, \quad |z| < 1,\label{eulerian_convergent_series}
\end{equation}
and if $|z| > 1$ then $A_k(z) = (-1)^{k+1} A_k\left(\frac{1}{z} \right)$.

\begin{proposition}
For $a, s, z \in \mathbb{C}$, $\Re(a) > 0$ and $z \notin [1, \infty)$ we have the following uniform asymptotic expansion for $\Phi(z,s,a)$
\begin{equation}
\Phi(z, s,a) = \frac{e^{\mu}}{e^{\mu} - z} \left(\frac{1}{a^s} + \sum_{k=2}^{K-1} (-1)^k \frac{P_k(s)}{k! a^{k+s}} r^k A_k\left(\frac{e^{\mu}}{z}\right)\right) +\varepsilon_K(z, s, a),\label{uae_lerch}
\end{equation}
where $r = z/(e^{\mu}-z)$.
\end{proposition}

\begin{proof}
We take $f(t) = (1-ze^{-t})^{-1}$ and $\mu = s / a$ in (\ref{uae_int}), where $\mu$ is the saddle point of the dominant part of the integral. Following closely the derivation in \cite{Ferreira2004}, we expand $f(t)$ at $t = \mu$ to obtain the Taylor expansion
\begin{equation}
f(t) = \sum_{k=0}^{\infty} a_k(\mu)(t-\mu)^k, \quad a_k(\mu) = (-1)^k\frac{e^{\mu}z^k}{ k!(e^{\mu}-z)^{k+1}} \sum_{j=0}^{k-1} \genfrac<>{0pt}{}{k}{j} \left(\frac{e^{\mu}}{z}\right)^j.
\end{equation}
After performing a few algebraic manipulations we obtain the final representation for the vanishing point expansion for $\Phi(z, s, a)$.
\end{proof}

We can clearly observe that for large values of $\Re(s)$ and $|z|$, the asymptotic convergence of the expansion improves. Furthermore, from a numerical perspective, moderate to large values of $\Re(s) <0$ permit the evaluation of $A_k(e^{\mu} /z)$ via the convergent series (\ref{eulerian_convergent_series}).

It remains to bound the error term in the expansion after truncation at $k=K-1$. Let us consider the $k$-th term of expansion (\ref{uae_lerch}) defined as 
\begin{equation*}
|t_k| = \left| \frac{P_k(s)}{k! a^{k+s}} \left(\frac{z}{e^{\mu}-z}\right)^k A_k\left(\frac{e^{\mu}}{z}\right)\right| \le \left| \frac{1}{k! a^{k+s}} \left(\frac{z}{e^{\mu}-z}\right)^k \right| |P_k(s)| \left|A_k\left(\frac{e^{\mu}}{z}\right)\right|.
\end{equation*}
A bound for the error term by comparison with a geometric series yields
\begin{equation}\label{uae_remainder_bound}
\left| \sum_{k=K}^{\infty} t_k \right| \le  \frac{|t_K|}{1-C}, \quad C = \left|\frac{t_{K+1}}{t_K}\right|,
\end{equation}
iff $C < 1$, where $t_K$ is the first omitted term in the expansion and 
\begin{equation*}
|\varepsilon_K(z, s, a)| \le \left|\frac{e^{\mu}}{e^{\mu} - z}\right| \frac{|t_K|}{1-C}.
\end{equation*}
In order to provide an effective upper bound for $|t_k|$, we compute two saddle point bounds for polynomials $P_k(z)$ and $A_k(z)$.

\begin{proposition}
For $k > 1$ and $z\in \mathbb{C}\setminus\{1\}$ the Eulerian polynomials satisfy the following bound
\begin{equation}
|A_k(z)| \le k! \left|(z-1) e^{\phi(t_0)}\right|,\label{eulerian_bound}
\end{equation}
where 
\begin{equation*}
t_0 = \frac{W(e^k k z) - k}{z-1} \quad \text{and} \quad \phi(t_0) = -k \log t_0 - \log \left(z - e^{(z-1)t_0})\right),
\end{equation*}
and $W(x)$ is the Lambert-$W$ function which solves $W(x)e^{W(x)} = x$.
\end{proposition}

\begin{proof}
An integral representation for the Eulerian polynomials is obtained after applying Cauchy's integral formula to the exponential generating function given by
\begin{equation*}
\sum_{k=0}^{\infty}A_k(z) \frac{t^k}{k!} = \frac{z-1}{z-e^{(z-1)t}} \Longrightarrow A_k(z) = \frac{k! (z-1)}{2\pi i} \oint \frac{t^{-k-1}}{z - e^{(z-1)t}} \mathop{dt},
\end{equation*}
which can be written in the form
\begin{equation*}
A_k(z) = \frac{k!(z-1)}{2\pi i} \oint \frac{e^{\phi(t)}}{t} \mathop{dt}, \quad \phi(t) = -k \log t - \log \left(z - e^{(z-1)t}\right).
\end{equation*}

We compute the saddle point of the integrand by solving the following equation
\begin{equation}
\phi'(t) = \frac{(z-1)e^{(z-1) t}}{z - e^{(z-1) t}} - \frac{k}{t}, \quad t_0 = \frac{W(e^k k z) - k}{z - 1}.
\end{equation}
The principal contribution of the saddle point bound is obtained by substituting $t_0$ into the integrand
\begin{equation*}
|A_k(z)| \le \frac{k!}{2\pi i} \left|(z-1) e^{\phi(t_0)}\right| \oint \frac{dt}{t} = k!  \left|(z-1) e^{\phi(t_0)}\right|.
\end{equation*}
Finally, by the residue theorem we obtain the result.
\end{proof}

A similar analysis is carried out for polynomials $P_k(z)$. The use of the generating function for generalized Laguerre polynomials gives the Cauchy-type integral representation
\begin{equation}
P_k(s) = k! L_k^{-k-s}(-s) = \frac{k!}{2\pi i} \int_{\mathcal{C}} (1-t)^{k+s-1} e^{ts/(1-t)} \frac{\mathop{dt}}{t^{k+1}},
\end{equation}
where $\mathcal{C}$ is a circle around the origin with a radius less than unity.

\begin{proposition} For $k > 1$ and $s \in \mathbb{C}\setminus\{0,1\}$ the polynomials $P_k(s)$ satisfy the following bound
\begin{equation}
\left|P_k(s)\right| \le k! \left|e^{\phi(t_0)} \right|,\label{p_bound}
\end{equation}
where 
\begin{equation*}
t_0 = \frac{\sqrt{k^2 + 4ks -2k + 1} + k + 1}{2(1-s)}
\end{equation*}
and
\begin{equation*}
\phi(t_0) = s \frac{t_0}{1-t_0} + (k+s-1) \log(1 - t_0) - k \log t_0.
\end{equation*}
\end{proposition}

\begin{proof}
A proof follows the steps presented previously.
\end{proof}

Combination of both bounds (\ref{eulerian_bound}) and (\ref{p_bound}) gives the final form for the error bound. The selection of the appropriate truncation point $K$ to achieve a desired level of precision is detailed in Section \ref{section_algorithmic}.

\subsection{Asymptotic expansion for large z}\label{sub_asymp_large_z}
A careful reader shall have noticed that none of the previous series expansions are suitable for arbitrarily large $\Re(z) > 0$. The expansion in this subsection complements the asymptotic expansion described in \cite{Ferreira2004} for $z \in \mathbb{C} \setminus [0, \infty)$, $\Re(a)>0$ and $\Re(s)>0$ for large $z$ and fixed $a$ and $s$
\begin{theorem}
For $(a, b, z) \in \mathbb{C}$ and $\Re(a) > 0$ we have an asymptotic expansion for large $a$ and $z$, and fixed $s$ is given by
\begin{align}\label{asymp_large_z}
&\Phi(z, s, a) \sim \frac{1}{2a^s} + \frac{(-\log(z))^{s-1}}{z^a} \Gamma(1-s, -a\log(z))\nonumber\\
&+ \frac{1}{2a^s} \left(\frac{2}{\log(z)} - \coth\left(\frac{\log(z)}{2}\right)\right) \nonumber\\ 
&+ \frac{1}{a^s} \sum_{k=1}^{\infty} \frac{(s)_k}{a^k (2\pi)^{k+1}} \left(\frac{1}{u^{k+1}} -\frac{\pi^{k+1}}{k!} \coth(\pi u)^{k-1}\csch(\pi u)^2 \mathcal{P}_k(\sech(\pi u)^2)\right),
\end{align}
where  $u = \frac{\log(z)}{2\pi}$ and $\mathcal{P}_k(x)$ are \textit{peak polynomials} \cite{Stembridge1997}.
\end{theorem}

\begin{proof}
We start from the integral representation (\ref{int3}). Application of the binomial theorem yields
\begin{equation*}
I = \Im \left( \int_0^{\infty} \frac{z^{-it}}{(1-it/a)^s}\frac{\mathop{dt}}{e^{2\pi t} - 1}\right) = \Im\left(\sum_{k=0}^{\infty} \frac{(s)_k}{k!}\left(\frac{i}{a}\right)^k \int_0^{\infty} \frac{z^{-it} t^k}{e^{2\pi t}-1} \mathop{dt}\right).
\end{equation*}

Let us focus on the inner integral $I_k$ defined as 
\begin{equation*}
I_k = \int_0^{\infty} \frac{z^{-it} t^k}{e^{2\pi t}-1} \mathop{dt} = \int_0^{\infty} \frac{z^{-it} t^k (1-e^{-t})}{(1-e^{-t})(e^{2\pi t}-1)} \mathop{dt}.
\end{equation*}
Noting that $(1-e^{-t}) / (e^{2\pi t}-1) = \frac{1}{2}(\coth(\pi t) - 1) (\sinh(t) -\cosh(t) + 1)$, we split the integral obtaining a closed form in terms of the Hurwitz zeta function
\begin{equation*}
I_k = k! \left(\frac{\zeta(k+1, i\log(z)/(2\pi))}{(2\pi)^{k+1}} - \frac{1}{(i\log(z))^{k+1}} \right) = \frac{k!}{(2\pi)^{k+1}}\zeta\left(k+1,1+i\frac{\log(z)}{2\pi}\right).
\end{equation*}
For $k=0$, $\zeta\left(k+1,1+i\frac{\log(z)}{2\pi}\right)$ has a pole, so we proceed as follows
\begin{equation*}
\Im\left(\int_0^{\infty} \frac{z^{-it}}{e^{2\pi t}-1} \mathop{dt}\right) = -\int_0^{\infty} \frac{\sin(\log(z)t)}{e^{2\pi t}-1}\mathop{dt} = \frac{1}{4} \left(\frac{2}{\log(z)} -\coth\left(\frac{\log(z)}{2}\right)\right).
\end{equation*}
Combining terms give us the asymptotic expansion for integral (\ref{int3})
\begin{equation*}
I \sim \frac{1}{4} \left(\frac{2}{\log(z)} - \coth\left(\frac{\log(z)}{2}\right)\right) + \sum_{k=1}^{\infty} \frac{(s)_k}{a^k (2\pi)^{k+1}} \Im\left(i^k \zeta\left(k+1,1+i\frac{\log(z)}{2\pi} \right) \right).
\end{equation*}

Hereinafter we use $u = \frac{\log(z)}{2\pi}$ to simplify notation. Let us define the terms $C_k(u)$ as
\begin{equation*}
C_k(u) = \Im\left(i^k \zeta\left(k+1,1+iu \right) \right)=\frac{i^{k+1}}{2} \left( (-1)^k \zeta\left(k+1,1-iu\right) - \zeta\left(k+1,1+iu \right)\right),
\end{equation*}
where we remove the imaginary part. In order to eliminate the computations on the complex plane for real $z$, we expand\footnote{We employ \texttt{FunctionExpand} in Mathematica \cite{Mathematica}.} $C_k(u)$ reducing compound arguments. The first five coefficients $c_k(u) = 2 C_k(u)$ are 
\begin{align*}
c_1(u) &= \frac{1}{u^2} - \pi^2 \csch(\pi u)^2,\\
c_2(u) &= \frac{1}{u^3} - \pi^3 \coth(\pi u)\csch(\pi u)^2,\\
c_3(u) &= \frac{1}{u^4} - \frac{\pi^4}{6} \left(4 \coth(\pi u)^2\csch(\pi u)^2 + 2\csch(\pi u)^4\right),\\
c_4(u) &= \frac{1}{u^5} - \frac{\pi^5}{24} \left(8 \coth(\pi u)^3\csch(\pi u)^2 + 16\coth(\pi u)\csch(\pi u)^4\right),\\
c_5(u) &= \frac{1}{u^6} - \frac{\pi^6}{120} \left(16 \coth(\pi u)^4\csch(\pi u)^2 + 88\coth(\pi u)^2\csch(\pi u)^4 + 16\csch(\pi u)^4\right).
\end{align*}
From the observation of previous coefficients, we state the following identity, which proof follows by induction
\begin{align*}
C_k(u) &= \frac{1}{2u^{k+1}} - \frac{\pi^{k+1}}{2k!} \sum_{j=0}^{\lceil k/2 \rceil - 1} P(k, j) \coth(\pi u)^{k-1-2j}\csch(\pi u)^{2(j+1)}\nonumber\\
&= \frac{1}{2u^{k+1}} - \frac{\pi^{k+1}}{2k!}  \coth(\pi u)^{k-1} \csch(\pi u)^2 \sum_{j=0}^{\lceil k/2 \rceil - 1}  P(k, j) \sech(\pi u)^{2j}\nonumber\\
& = \frac{1}{2u^{k+1}} - \frac{\pi^{k+1}}{2k!}  \coth(\pi u)^{k-1} \csch(\pi u)^2 \mathcal{P}_k(\sech(\pi u)^2).
\end{align*}
where $P(k, j)$ denotes the number of permutations of $k$ numbers with $j$ peaks, also known as \textit{peak number} or \textit{pk-number}, and $\mathcal{P}_k(x)$ a \textit{pk-polynomial}.
\end{proof}

Peak numbers $P(k, j)$ give the sequence \href{https://oeis.org/A008303}{A008303} of the OEIS \cite{OEIS}. For $k \ge 1$ and $0 \le j \le k$, we have a functional recursion generating a triangular array
\begin{equation}\label{peak_number_recurrence}
P(k, j) = 2(j+1) P(k-1, j) + (k-2j) P(k-1, j-1).
\end{equation}
Note that $P(k, j) = 0$ for $j \ge k/2$ and therefore $\deg P_k(x) =\lceil k/2 \rceil - 1$. Peak polynomials are given by the generating function for peak numbers $P(k,j)$.
\begin{equation*}
\mathcal{P}_k(x) = \sum_{j=0}^{\lceil k/2 \rceil - 1} P(k, j) x^j.
\end{equation*}
Note that for values of $|x| \to 1$ we can estimate its magnitude by the finite sum of peak numbers, since $\sum_{j=0}^{\lceil k/2 \rceil - 1} P(k, j) = k!$, hence
\begin{equation}
|\mathcal{P}_k(x)| \sim k!, \quad |x|\to 1.
\end{equation}

The bivariate exponential generating function can be defined as in \cite{Fewster2014}, \cite{Zhuang2016}
\begin{align*}
\sum_{k=0}^{\infty} \mathcal{P}_k(p) \frac{x^k}{k!} &= 1 - \frac{1}{p} + \frac{\sqrt{p-1}}{p} \tan \left(x \sqrt{p - 1} + \arctan(1/\sqrt{p-1})\right)\\
&= \frac{\sqrt{1-p} \cosh(x \sqrt{1-p})}{\sqrt{1-p} \cosh(x \sqrt{1-p}) - \sinh(x \sqrt{1-p})} = \frac{1}{\sqrt{1-p} \coth(x\sqrt{1-p}) - 1}
\end{align*}
As customary in analytic combinatorics, application of Cauchy's integral formula to the bivariate exponential generating function gives
\begin{equation*}
\mathcal{P}_k(p) = \frac{k!}{2\pi} \int_0^{2\pi} \frac{e^{-ikt}}{\sqrt{1-p} \coth(e^{it}\sqrt{1-p}) - 1} \mathop{dt}.
\end{equation*}

A remarkable result from the theory of enriched $P$-partitions is the functional relation between peak polynomials and Eulerian polynomials stated in \cite{Stembridge1997}
\begin{equation}\label{func_rela_eul_peak}
\mathcal{P}_k\left(\frac{4x}{(1+x)^2}\right) = \frac{2^{k-1}}{(1+x)^{k-1}}A_k(x),
\end{equation}
which allows us to use the upper bound in (\ref{eulerian_bound}) to estimate the truncation point in (\ref{asymp_large_z}). Furthermore, a good asymptotic estimate of $P_k(x)$ for large order $k$ can be derived from a Mittag-Leffler type decomposition of Eulerian polynomials \cite{Costin2009}:
\begin{equation*}
A_k(z) = C(k, z) \left(\frac{1}{\log(z)^{k+1}} + \sum_{j=1}^{\infty} \frac{1}{\left(\log(z) + 2\pi i j\right)^{k+1}} + \frac{1}{\left(\log(z) - 2\pi j k\right)^{k+1}}\right),
\end{equation*}
where 
\begin{equation*}
C(k, z) = \frac{e^{\pi i (k-1)}(1-z)^{k+1} k!}{z}.
\end{equation*}
Taking the prefactor of the expansion and applying the functional relation (\ref{func_rela_eul_peak}) we have
\begin{equation*}
\mathcal{P}_k(x) \sim \frac{2^{k-1} k! e^{\pi i (k-1)}(1-u)^{k+1}}{(1+u)^{k-1} u \log(u)^{k+1}}, \quad u = \frac{2-x-2\sqrt{1-x}}{x}, \quad k \to \infty.
\end{equation*}

\section{Algorithmic details and implementation}\label{section_algorithms}\label{section_algorithmic}

In this section we discuss in detail the implementation aspects and several proposed heuristics easy to evaluate while being effective in practice. All algorithms described are implemented in Python\footnote{\url{https://sites.google.com/site/guillermonavaspalencia/software/lerch.py}} using the mpmath library for arbitray-precision floating-point arithmetic \cite{mpmath} with GMPY2, which supports integer and rational arithmetic via the GMP library \cite{GMP} and real and complex arithmetic by the MPFR \cite{MPFR} and MPC \cite{MPC} libraries.

As it is well-known, numerical evaluation of special functions requires the use of several methods of computation to cover the whole regime of the parameters. We aim to sketch the building blocks of a basic algorithm, which have been tested to work reasonable well for most cases, but we do not dare to claim that it will cover the whole function's domain optimally. For those cases either not covered by current series expansions or prone to numerical instability, we select numerical complex integration, which serves as a backup method.

\subsection{Evaluation of L-series}

The L-series of the form (\ref{l_series}) are in general difficult to accelerate due to the non recursive scheme of computation. In order to employ common acceleration techniques such as parallelization, the determination of the optimal truncation level is crucial. As described in the previous section, a bound for the remainder term of the L-series can be constructed as follows
\begin{equation}
\left| \sum_{k=K}^{\infty} \frac{z^k}{(k+a)^s} \right| \le \frac{|z|^K }{|(K+a)^s| (1-C_K(z,s,a))} \le \frac{|z|^K }{|(K+a)^s| (1-|z|)}, 
\end{equation}
where
\begin{equation}
C_K(z,s,a) = \left|\frac{z (K+a)^s}{(K+1+a)^s}\right|, \quad \lim_{K\to\infty} C_K(z,s,a) = |z|.
\end{equation}

The required number of terms $K$ to obtain a result with $P$-bit precision can be obtained by performing a simple linear search, which is generally sufficient to target an absolute error of about $2^{-P}$. However, a more efficient approximation of $K$ is yielded by solving the following equation with the first omitted term, $z^K (K+a)^{-s} = 2^{-P}$ for $K$. The first solution in closed form is given by
\begin{equation*}
K = -\frac{sW(\phi(z, s, a, P)) + a\log(z)}{\log(z)}, \quad \phi(z, s, a, P) = -\frac{(2^{-P} z^{a})^{-1/s} \log(z)}{s}.
\end{equation*}

We distinguish two different approximations for $K$, denoted as $\hat{K}$, depending on $\Re(s)$. For $\Re(s) > 0$
\begin{equation*}
\hat{K} = \left[\left|\frac{\Re(s) W_0(\phi_1(z, s, a, P)) + |a|\log(|z|)}{\log(|z|)}\right|\right],
\end{equation*}
where $[x]$ denotes the nearest integer function and
\begin{equation*}
\phi_1(z, s, a, P) = -\frac{(2^{-P} |a|^{-\Re(s)-1} |z|^{|a|})^{-1/(\Re(s)+1)} \log(|z|)}{\Re(s)+1}.
\end{equation*}
For $\Re(s) < 0$
\begin{equation*}
\hat{K} = \left[\left|\frac{\Re(s) W_{-1}(\phi_2(z, s, a, P)) + |a|\log(|z|)}{\log(|z|)}\right|\right],
\end{equation*}
where
\begin{equation*}
\phi_2(z, s, a, P) = -\frac{(2^{-P} |a|^{-\Re(s)} |z|^{|a|})^{-1/\Re(s)} \log(|z|)}{\Re(s)}.
\end{equation*}

Given that $\phi_1$, $\phi_2 \in \mathbb{R}$, we use the principal branch $W_0(z) = W(z)$ when $\Re(s) > 0$, since $\phi_1 > 0$ and the branch $W_{-1}(z)$ for $\Re(s) <0$, since $\phi_2 \in (-1/e, 0)$. Remember that $W(x)$ is two-valued for $-1/e \le x < 0$. Numerical tests suggest that these approximations for choosing $K$ are sufficient to obtain good estimates of the required number of terms. We note that $\hat{K}$ can be computed using 53-bit machine floating-point arithmetic.

Several heuristics are implemented to compensate catastrophic cancellation for cases when $\Re(z) < 0$ and/or $\Re(s) < 0$. In particular, for $\Re(s) <0$ we increase the working precision $P^{W} = P + \lfloor P /3 \rfloor + [-\Re(s)]$. On the other hand, for the case $\Re(s) > 1$ and $z \in \mathbb{R}_{<0}$ we employ the linear acceleration methods for alternating series described in \cite{Cohen2000}. This method is used when $\hat{K} > 1.2 [1.31 D]$, where $D$ is the precision digits. For all other cases, we add up to 20 guard bits to the working precision.

The computation of the L-series is particularly simple to parallelize by assigning a block of size $k: k \le N$ to each thread. This parallelization scheme is implemented using the \texttt{multiprocessing} module in Python. Based on experiments, parallelization provides a significant speedup factor for $\hat{K} > 1024$ or $P \ge 1024$ bits.

We remark that L-series converges rather slowly when $|z| \to 1$. It is possible to employ convergence acceleration techniques to obtain an efficient evaluation of the Lerch transcendent; see the application of combined nonlinear-condensation transformation in \cite{Jentschura1999}. Alternatively, the Euler-Maclaurin formula is also convenient for those cases, as shown later.

\subsection{Evaluation of Euler-Maclaurin formula}

\subsubsection{Evaluation of the error bound}

For a precision of $D$ digits, we choose $N = \lfloor D / 3 \rfloor$. For large $\Re(a) > 0$ we choose $N=0$ if the following condition is satisfied
\begin{equation*}
\Re(a) > |\Re(s)| + |\Re(z)| + D.
\end{equation*}

The number of terms $M$ can be effectively approximated by solving $t_k = 2^{-P}$, where $P$ is the precision in bits and $t_k$ is the asymptotic estimate in (\ref{tk_asymp}), which yields
\begin{equation*}
M \sim \left[\frac{1}{2}\left|\frac{\log(2^{-P-1} \log(z))}{\log(2\pi) - \log(\log(z))}\right|\right].
\end{equation*}
This is a near-optimal approximation at high-precision. In practice, the asymptotic estimate of $M$ is used for $P \ge 500$, otherwise we use the heuristic $M = N + \lfloor P / 3 \rfloor$. There is a unavoidable trade-off when choosing $N$ and $M$, large values results in catastrophic cancellation since the L-series might be unstable, especially for $z \in \mathbb{C}\setminus \mathbb{R}$, but reduces the number of terms $M$, therefore the time spent computing Bernoulli numbers, which represents a significant amount of the total computation time.

We can evaluate the coefficient in the error bound (\ref{remainder_bound}) using a recurrence. Computation of $Q(k+1-2M-s, -(a+N)\log(z))$ only requires the initial value $Q(1-s, -(a+N)\log(z))$, which can be computed re-using $\Gamma(1-s, -(a+N)\log(z))$ in (\ref{em_tail}). Subsequent terms can be computed at lower precision via the recurrence
\begin{equation*}
Q(a, z) = Q(a-1, z) + \frac{e^{-z} z^{a-1}}{\Gamma(a)}, \quad a \in \mathbb{C}\setminus\mathbb{Z}^-,
\end{equation*}
or 
\begin{equation*}
\Gamma(a, z) = (a-1)\Gamma(a-1,z) + e^{-z} z^{a-1}, \quad a \in \mathbb{C}.
\end{equation*}
A recurrence for the rest of terms in the coefficients is trivial. The direct evaluation of the recurrence requires $\mathcal{O}(M)$ arithmetic operations, therefore the associated computational cost is not negligible for very large $M$, as previously mentioned. 

For large $M$ we might use asymptotic estimates to reduce the complexity, for example the case $M > |(a+N)\log(z)|$
\begin{align}
|R| &\sim \frac{4}{(2\pi)^{2M}} \left|(-2M+1-s)_{2M} E_{2M + s}(-(a+N)\log(z)) \frac{(a + N)^{-2M - s + 1}}{z^{a+N}}\right| \nonumber\\
& \sim \frac{4}{(2\pi)^{2M}} \left| \frac{(-2M+1-s)_{2M}}{2M+s} \frac{(\log(z)/z)^{a+N}}{(a+N)^{2M+s-1}}  \right|, \label{approx_M_1}
\end{align}
where $E_{\nu}(z)$ is the generalized exponential integral \cite{Navas-Palencia2018a} and we take the asymptotic estimate $E_{\nu}(z) \sim e^{-z}/\nu$ as $\nu \to \infty$. 

For $M \sim |(a+N)\log(z)|$ and $|s| \ll M$, we use the first order estimate of the Franklin-Friedman expansion for $U(a, b, z)$ in \cite{Navas-Palencia2018b} given by
\begin{equation*}
U(-2M, -2M+1-s, -(a+N)\log(z)) \sim \left(1 + \frac{2M}{(a+N)\log(z)}\right)^{-s}((a+N)\log(z))^{2M},
\end{equation*}
replacing it in (\ref{remainder_bound}) and after observation that the remaining integral is expressible in terms of the incomplete gamma function we obtain 
\begin{equation}\label{approx_M_2}
|R| \sim \frac{4}{(2\pi)^{2M}} \left|\frac{(-\log(z))^{2M+s-1}}{3^{\min(\Re(s), 0)}z^a} \Gamma(1-s, -(a+N)\log(z) \right|.
\end{equation}

Table \ref{table_bound_remainder_em} shows a few examples when $|z| < 1$, otherwise the integral (\ref{remainder}) is not well defined. Approximations (\ref{approx_M_1}) and (\ref{approx_M_2}), although being quite simple, might be used to compute a crude estimate of the magnitude of the remainder in reasonable time.
\begin{table}[H]
	\centering
	\scalebox{0.85}{
\begin{tabular}{ccccccccc}
\hline
$z$ & $s$ & $a$ & $N$ & $M$ &  (\ref{remainder}) & (\ref{remainder_bound}) & (\ref{approx_M_1}) & (\ref{approx_M_2})\\
\hline
	0.8 & 3.2 & 10.5 & 6 & 15 & $7.5e{-30}$ & $1.5e{-28}$ & $1.3e{-38}$ & $1.1e{-47}$\\
	0.5 + 0.2i & -30.2 -i & 10.5 +5i & 10 & 40 & $4.8e{-24}$ & $1.0e{-22}$ & $8.1e{-37}$ & $3.2e{-20}$\\
	0.5 + 0.2i & -30.2 -i & 100.5 +5i & 100 & 2000 & $2.5e{+302}$ & $5.6e{+302}$ & $3.9e{+282}$ & $-$\\
	0.5 + 0.7i & -3.2 +10i & 10.5 + 5i & 250 & 300 & $1.6e{-503}$ & $4.8e{-502}$ & $1.5e{-503}$ & $5.4e{-496}$\\
	0.5 + 0.7i & 30.2 +10i & 10.5 + 10.5i & 600 & 700 & $9.9e{-1240}$ & $1.2e{-1238}$ & $7.5e{-1216}$ & $3.8e{-1264}$\\
	\hline\\
	\end{tabular}}
	\caption{Effectiveness of bound (\ref{remainder_bound}) in error term of the Euler-Maclaurin formula.}
	\label{table_bound_remainder_em}
\end{table}

\subsubsection{Evaluation of the tail}
A more interesting form of the tail (\ref{em_tail}) is obtained by applying Kummer's transformation to $U(-2k + 1, -2k+2-s, -(a+N)\log(z))$, thus
\begin{equation*}
T = z^N \bigg(\frac{1}{2(a+N)^s} + (-\log(z))^s\sum_{k=1}^{M} \frac{B_{2k}}{(2k)!} (-\log(z))^{2k-1} U(s, s+2k, -(a+N)\log(z))\bigg).
\end{equation*}
For this particular case, $U(a, b, z)$ reduces to a polynomial in $-(a+N)\log(z)$ of degree $2k-1$, indeed expressible in terms of generalized Laguerre polynomials, given by
\begin{equation*}
U(s, s+2k, -(a+N)\log(z))) = (-(a+N)\log(z))^{-s} \underbrace{\sum_{j=0}^{2k-1} \binom{2k - 1}{j} \frac{(s)_j}{(-(a+N)\log(z))^j}}_{T_2^{(k)}}.
\end{equation*}
Terms $T_2^{(k)}$ can be constructed using a linear holonomic recurrence equation. Let us define the constants expressions $p$ and $q$
\begin{equation*}
p = s - (a + N)\log(z)), \quad q = -\frac{1}{(a+N)\log(z)}.
\end{equation*}
The sequence of terms $T_2^{(k)}$ satisfy the recurrence equation
\begin{align*}
T_1^{(k)} &= [p T_2^{(k-1)} + (2k - 3) (T_2^{(k-1)} - T_1^{(k-1)})] q\\
T_2^{(k)} &= [p T_1^{(k)} + (2k - 2) (T_1^{(k)} - T_2^{(k-1)})] q
\end{align*}
for $k \ge 2$, with initial values 
\begin{equation*}
T_1^{(1)} = 1, \quad T_2^{(1)} = 1 - \frac{s}{(a+N)\log(z)}.
\end{equation*}
A matrix form for $k \ge 2$ is defined as 
\begin{equation*}
\begin{pmatrix}
T_1^{(k)} \\
T_2^{(k)} \\
T_2^{(k-1)}
\end{pmatrix} = 
\begin{pmatrix}
q(p + 2k -3) & q(3-2k) & 0\\
q (2 - 2k) & 0 & q(p+2k-2)\\
1 & 0 & 0 \\
\end{pmatrix}
\begin{pmatrix}
T_2^{(k-1)} \\
T_1^{(k-1)} \\
T_1^{(k)}
\end{pmatrix}
\end{equation*}
or simply
\begin{equation*}
\begin{pmatrix}
T_2^{(k)} \\
T_1^{(k)} \\
\end{pmatrix} = q
\begin{pmatrix}
k(4k + 4p -12) + (p-5)p + 8 & k(-4k -2p + 10) + 3p -6\\
3-2k & p+2k-3\\
\end{pmatrix}
\begin{pmatrix}
T_2^{(k-1)} \\
T_1^{(k-1)} \\
\end{pmatrix}.
\end{equation*}

The complexity of the recurrence scheme is $\mathcal{O}(M P)$ and requires a small temporary storage. Note that a matrix recurrence for the sequence of coefficients $T_*^{(k)}$ is suitable in a binary splitting scheme. The previous analysis results in a more tractable expression for the tail $T$
\begin{equation}
T = \frac{z^N}{(a+N)^s} \bigg(\frac{1}{2} + \sum_{k=1}^{M} \frac{B_{2k}}{(2k)!} (-\log(z))^{2k-1} T_2^{(k)}\bigg).
\end{equation}

Now the terms of tail sum $T$ satisfy a recurrence equation except for the multiplication by Bernoulli numbers. The Bernoulli numbers are cached for repeated evaluation, but computing them the first time at very high precision is time-consuming. We do not attempt to improve current implementations but rather rely on the algorithm implemented in mpmath, which automatically caches Bernoulli numbers $B_n$ when $n < 3000$ for multiple evaluations. For larger values of $n$ the connection to Riemann zeta function $\zeta(n)$ is used. Many recursive algorithms for computing $B_0, \ldots, B_n$ such as $B_n = - \sum_{k=0}^{n-1} \frac{B_k}{k!(n-k+1)}$ require $\mathcal{O}(n^2)$ arithmetic operations. As an alternative, an algorithm based on recycling terms in the Riemann zeta function series expansion, which also have cubic complexity, is implemented in \cite{arb}.

\subsection{Evaluation of asymptotic expansions}

The main drawback of the asymptotic expansions in (\ref{uae_lerch}) and (\ref{asymp_large_z}) is the difficulty of computing a large number of Eulerian and peak polynomials efficiently.
Computing the first $k$ Eulerian polynomials simultaneously can be performed by using the recursion in (\ref{eulerian_recursion}). Thus, given $A_{0}(z), \ldots,A_{k-1}(z)$, we can compute $A_k(z)$ in $\mathcal{O}(k)$ arithmetic operations, and noting that $A_k(z)$ has $\mathcal{O}(k \log k)$ bits from (\ref{eulerian_bound}), the algorithm needs $\mathcal{O}(k^{3+ o(1)})$ bit operations and requires $\mathcal{O}(k^2 \log k)$ space to store previous $A_j(z), j < k$. For example, using a straightforward implementation, we compute $A_0(2), \ldots, A_{1000}(2)$ at 333-bit precision in 1.51 seconds on a 2.6 GHz Intel i7 processor.  

To compute $k$ Eulerian polynomials in time complexity $\mathcal{O}(k^{2+ o(1)})$ we might apply a \textit{multisectioning} scheme to the bivariate exponential generating function. Alternatively, it is possible to recycle terms of the sum (\ref{eulerian_convergent_series}) to speedup multievaluation, considering that terms $j^k z^j$ can be optimized to only compute binary exponentiation when $j$ is prime and multiplication otherwise. The required number of terms is approximated by solving $J^k z^J = 2^{-P}$,
\begin{equation}
J^* \approx \left[\left| \Re\left(-\frac{k W_{-1} \left(-\frac{(2^{-P})^{1/k} \log(|z|)}{k}\right)}{\log(|z|)}\right)\right|\right].
\end{equation}

For large $k$ the size of $J^*$ growth rapidly, therefore it is convenient to apply asymptotic faster methods such as the Mittag-Leffler type decomposition introduced in \S 2.3, which acts as an asymptotic expansion. For $z \in \mathbb{R}$, two optimizations can be implemented:
\begin{equation*}
A_k(z) = \frac{(z-1)^{k+1} k!}{z} \Re\left(\frac{1}{\log(z)^{k+1}} + 2\sum_{j=1}^{\infty} \frac{1}{\left(\log(z) + 2\pi i j\right)^{k+1}}\right), \quad z \in \mathbb{R}^+,
\end{equation*}
\begin{equation*}
A_k(z) = \frac{(z-1)^{k+1} k!}{z} \Re\left(\frac{1}{\log(z)^{k+1}} +\sum_{j=1}^{\infty} \frac{1}{\left(\log(z) - 2\pi i j\right)^{k+1}}\right),\quad z \in \mathbb{R}^-.
\end{equation*}
To compute a single $A_{1000}(1/5)$ at 333-bit, the power series requires $J^* = 5545$ whereas the Mittag-Leffler decomposition only needs $J^* = 59$ terms, hence a complete algorithm shall combine the iterative computation via the three-term recurrence and the asymptotic expansion as $k\to \infty$.

Like other orthogonal polynomials, polynomials $P_k(\lambda)$ in (\ref{tricomi_carlitz}), which are strongly related to Tricomi-Carlitz polynomials, satisfy three-term recurrence,
\begin{equation*}
P_0(\lambda) = 1, \quad P_1(\lambda) = 0, \quad \text{and} \quad P_{k+1}(\lambda) = k(P_k(\lambda) + \lambda P_{k-1}(\lambda)), \; k > 1.
\end{equation*}
Given the complexity of computing a large number of Eulerian polynomials, the cost of the three-term recurrence is almost negligible.

The truncation level $K$ in (\ref{uae_remainder_bound}) is computed at lower precision via linear search and $C$ is estimated as
\begin{equation*}
C = \left|\frac{t_{K+1}}{t_K}\right| \sim \left|\frac{z}{z-e^{\mu}}\frac{(s+K) (z-1)}{a\log(z)}\right|, \quad k \to \infty,
\end{equation*}
from which we obtain an estimate of the number of terms
\begin{equation*}
K_{max} \approx \left| \frac{a \log(z) (z-e^{\mu})}{z(z-1)} - s\right|.
\end{equation*}

Computation of peak polynomials is carried out using the generating function for peak numbers for moderate $k$, which evaluation only involves roughly half of the terms $k/2$ compared to the Eulerian polynomials. For example, computing the triangular array for the first 1000 peak numbers using recurrence (\ref{peak_number_recurrence}) takes 1.55 seconds.
For larger $k$ the functional relation with the Eulerian polynomial is applied. 
A trickier aspect of the asymptotic expansion (\ref{asymp_large_z}) is to determine the optimal truncation $K$. The coefficients $c_k(u)$ behave as
\begin{equation*}
|c_k(u)| = \left|\frac{1}{u^{k+1}} - \frac{\pi^{k+1}}{k!}  \coth(\pi u)^{k-1} \csch(\pi u)^2 \mathcal{P}_k(\sech(\pi u)^2)\right| \sim \frac{1}{|1 -i u|^k},
\end{equation*}
as $k \to \infty$. Given the ratio of convergence of the asymptotic expansion, we can estimate the maximum number of terms $K_{max}$, thus the maximum attainable accuracy as follows
\begin{equation*}
\left|\frac{t_{K+1}}{t_K}\right| \sim \left|\frac{(s+K)}{a 2 \pi (1 -i u)}\right|, \quad k \to \infty, \quad K_{max} \approx |a(2\pi - i \log(z)) - s|.
\end{equation*}
It remains to estimate the required numbers of terms $K$ to target $P$-bit accuracy, which is approximated heuristically and subsequently refined via linear search using
\begin{equation*}
K \approx \left[\left|\frac{\varphi(s)}{W_{-1}(\frac{\varphi(s)}{a 2 \pi \log(z)})}\right|\right],
\end{equation*}
where $\varphi(s) = -\log(2) P - \log(1+s)$. In fact, we slightly increase $K$ by a factor $\approx 1.2$, which works well in practice. Hence, we can evaluate the asymptotic expansion as long as $K \le K_{max}$ to target an absolute error of $2^{-P}$.

\subsection{Numerical integration}

The current implementation in mpmath computes the Lerch transcendent via numerical integration using the double-exponential method for the integral representation (\ref{int2}) employing the \texttt{quad} function. We choose numerical integration for $\Re(a) > 0$ as a backup method when computation by aforementioned methods is not satisfactory. Note that a few optimizations are possible for real parameters by rewriting the integrand, for example, the integral representation (\ref{int3}) when $z > 0$ and $a > 0$ or
\begin{equation*}
I = -\Im \left( \int_0^{\infty} \frac{(1-it/a)^s z^{it}}{(a^2 + t^2)^{s/2} (1+t^2/a^2)^{s/2}}\frac{\mathop{dt}}{(e^{2\pi t} - 1)} \right), \quad z, s, a\in \mathbb{R}, \, z > 0,
\end{equation*}
both integral representations avoiding evaluation of trigonometric functions. The computation at high-precision, say 1000 digits onwards, is generally costly compared to asymptotic methods, therefore this is the method of choice when only strictly indispensable.

\section{Benchmark}

In this Section, we benchmark our implementation to current state-of-the-art software supporting evaluation of the Lerch transcendent function to arbitrary-precision. Tests were conducted on an Intel(R) Core(TM) i7-6700HQ CPU at 2.60GHz, using up to 4 cores for parallel mode, running Ubuntu Linux. We compare the computing times of Mathematica 10.4 and mpmath 1.0.0 using functions \texttt{Timing[]} and \texttt{time.perf\_counter()}, respectively. For mpmath we set the precision in bits $p$ using \texttt{mpmath.mp.prec = p}, whereas for Mathematica the desired level of precision in digits $d$ is set with \texttt{N[..., d]}, applying the conversion factor $d = \lfloor 0.301 p \rfloor$. To assess the correctness of our implementation, we compare to Mathematica at higher precision since it is frequently faster and more reliable than mpmath. Note, however, that Mathematica attempts to achieve $d$ digits of precision might fail unexpectedly, therefore we check the consistency of results at increasing levels of precision. 

The following tables show timing results to compute the Lerch transcendent for various regimes of the parameters and argument, varying the level of precision. We remark that Mathematica and mpmath use GMP internally, so timing measurements are directly comparable.

Table \ref{table_lerch_em} shows the performance of the Euler-Maclaurin formula (\ref{int_u_expan_final_form}) for small $z$ and moderate values of $s$ and $a$. The Euler-Maclaurin formula is implemented in a loop manner checking the level of cancellation at each iteration and increasing the working precision accordingly to correct it. Hence, a better estimation of the total amount of cancellation would reduce the computation time considerably. However, as we see for these cases, both Mathematica and mpmath are regularly an order of magnitude slower. Furthermore, as previously noted, larger values of $|a|$ improve the convergence of the series, reducing significantly the number of terms $N$ and $M$; Table \ref{table_lerch_em_N_M} shows the metrics corresponding to the last iteration.

\begin{table}[H]
	\centering
	\scalebox{0.85}{
\begin{tabular}{cccccc}
\hline
$\Phi(z,s, a)$ & bits & mpmath & Mathematica & Euler-Maclaurin & Parallel\\
\hline
	& 64 &  0.096 (0.139) & 0.313 & 0.008 (0.013) & -\\
	$z = 2.5 + 1.5i$ & 333 & 1.09 (1.21) & 0.672 & 0.041 (0.089) & -\\
	$s = 1.25 + 2i$ & 1024 & 8.39 (9.38) & 3.14 & 0.128 (0.233) & -\\
	$a = 3.5 + 5i$ & 3333 & 154.4 (161.1) & 27 & 1.64 (2.37) & -\\
	& 10000 &  1564.6 & 438 & 25.04 (33.33) &  19.06 (27.32) \\
	\hline
	& 64 & 0.263 (0.295) & 0.047 & 0.016 (0.039) & -\\
	$z = 2.5 +7.5i$ & 333 & 1.79 (1.98) & 0.250  & 0.110 (0.250) & - \\
	$s =-50.25 +10i$ & 1024 & 6.59 (7.05) & 1.58 & 0.72 (1.49) & -\\
	$a = 1.5 -i$ & 3333 & 133.6 (135.1) & 21.66 & 11.83 (16.64) & 8.72 (13.22)\\
	& 10000 & 1453 & 409.1 & 190.3 (224.5) & 153.2 (196.5) \\
	\hline
	& 64 & 0.253 (0.324) & 0.031 & 0.013 (0.022) & -\\
	$z = 2.5 +0.5i$ & 333 & 1.79 (1.92) & 0.265 & 0.058 (0.104) & -\\
	$s =-100.25 +10i$ & 1024 & 12.66 (13.94)  & 7.72 & 0.298 (0.685) & -\\
	$a = 100.5 -10i$ & 3333 & 120.7 (127.1) & 83.14 & 2.97 (4.19) & -\\
	& 10000 & 1500.6 & $>$ 1800  & 24.65 (32.89) & 18.88 (28.34)\\
	\hline\\
	\end{tabular}}
	\caption{Time (in seconds) to compute $\Phi(z,s,a)$ with moderate values of $z$, $s$ and $a$ to 64, 333, 1024, 3333 and 10000 bits of precision. First evaluation pre-computing Bernoulli numbers within parentheses. Maximum time 1800 seconds.}
	\label{table_lerch_em}
\end{table}

Current implementation does not incorporate complexity-reducing methods for the evaluation of the tail but simply uses the recurrence scheme and only includes optional parallelization of the truncated L-series, thus limiting the observable improvement by activating the parallel mode.

\begin{table}[H]
	\centering
	\scalebox{0.85}{
\begin{tabular}{ccccc}
\hline
$\Phi(z,s, a)$ & bits & N & M & $P^{W}$(bits)\\
\hline
	& 64 &  7 & 32 (H) & 86\\
	$z = 2.5 + 1.5i$ & 333 & 39 & 172 (H) & 400\\
	$s = 1.25 + 2i$ & 1024 & 123 & 247 (A) & 1229\\
	$a = 3.5 + 5i$ & 3333 & 401 & 805 (A) & 4000\\
	& 10000 & 1203 & 2416 (A) & 12000\\
	\hline
	& 64 & 28 & 122 (H) & 84\\
	$z = 2.5 +7.5i$ & 333 & 127 & 402 (A) & 383\\
	$s =-50.25 +10i$ & 1024 & 364 & 1146 (A) & 1094\\
	$a = 1.5 -i$ & 3333 & 1115 & 3503 (A) & 3346\\
	& 10000 & 3193 & 10032 (A) & 10419\\
	\hline
	& 64 & 10 & 46 (H) & 108\\
	$z = 2.5 +0.5i$ & 333 & 52 & 97 (A) & 528\\
	$s =-100.25 +10i$ & 1024 & 156 & 285 (A) & 1561\\
	$a = 100.5 -10i$ & 3333 & 480 & 876 (A) & 4790\\
	& 10000 & 1214 & 2216 (A) & 12108\\
	\hline\\
	\end{tabular}}
	\caption{Number of terms $N$, $M$ in the Euler-Maclaurin expansion and working precision $P^{W}$ for Euler-Maclaurin cases. (A) and (H) indicate the method used to estimate $M$, asymptotic and heuristic, respectively.}
	\label{table_lerch_em_N_M}
\end{table}

Table \ref{table_lerch_series} assesses the performance of the L-series implementation and its particular cases for $|z| < 1$ and small values of $s$ and $a$. Due to the performance gap between Mathematica and mpmath (mpmath only implements numerical integration), only the former is used for benchmarking on subsequent tests. Results show that our implementation is comparable to Mathematica (presumably evaluating the same L-series) at lower precision and it is found to be surprisingly faster at higher precision\footnote{We guess the poor performance is due to incorrect error tracking, which overestimates the required working precision.}. Moreover, we observe that our parallelization scheme achieves speedup ratios close to theoretical maximum, which apparently is  not implemented in Mathematica. Interestingly, the Euler-Maclaurin formula should be the preferred algorithm at low-medium precision when $|z| \sim 1$. 
\begin{table}[H]
	\centering
	\scalebox{0.85}{
\begin{tabular}{cccccc}
\hline
$\Phi(z,s, a)$ & bits & Mathematica & L-series & Parallel & Euler-Maclaurin\\
\hline
	           & 64   & 0.0011  & 0.0011 & - & -\\
	$z = 1/4$  & 333  & 0.0067  & 0.0055 & - & -\\
	$s = 10/4$ & 1024 & 0.0339  & 0.0155 & - & -\\
	$a = 20/7$ & 3333 & 0.7313  & 0.0660 & 0.0431 & -\\
	           & 10000 & 21.406 & 0.4885 & 0.1513 & -\\
	\hline
	           & 64   & 0.0015  & 0.0069 & - & 0.0031 (0.0053)\\
	$z = 9/10$ & 333  & 0.0109  & 0.0607 & 0.0429 & 0.0153 (0.0255)\\
	$s = 10/4$ & 1024 & 0.0984  & 0.2165 & 0.1052 & 0.0166 (0.0222)\\
	$a = 20/7$ & 3333 & 1.8843  & 1.0422 & 0.3281 & 0.0953 (0.1197)\\
	           & 10000 & 39.937  & 8.6969 & 2.4881 & 2.7331 (3.1578)\\
	\hline
	           & 64   & 0.0031  & 0.0020 & - & -\\
	$z =-6/10$ & 333  & 0.0156  & 0.0083 & - & -\\
	$s = 10/4$ & 1024 & 0.1422  & 0.0243 & - & -\\
	$a = 20/7$ & 3333 & 3.0922  & 0.0983 & - & -\\
	           & 10000 & 21.2969 & 0.7500 & - & -\\
	\hline
	\end{tabular}}
	\caption{Time (in seconds) to compute $\Phi(z,s,a)$ for small argument $|z|$.}
	\label{table_lerch_series}
\end{table}

The third example assesses the performance of the series acceleration technique for alternating series, which while it is hardly parallelizable, it is consistently faster than Mathematica for all tested instances.

Table \ref{table_large_z} shows the time to compute the Lerch transcendent using the asymptotic expansion (\ref{asymp_large_z}) for large $z$ and $a$. 
The optimal truncation of the first test is $K_{max} = 1598$, limiting the evaluation at 3333 bits of precision, which would require $K = 2767$. The optimal truncation for the second test is $K_{max} = 22297$ requiring up to 1952 terms at 10000 bit of precision. As noted, the time spent on the computation of a large number of peak polynomials accounts for a significant amount of the total time, therefore a more sophisticated and efficient algorithm would be needed at higher precision. On the other hand, for multiple evaluations, peak numbers can be cached same as Bernoulli numbers.

Numerical experiments show a performance deterioration of the Euler-Maclaurin formula as $z$ increases due to catastrophic cancellation, therefore its use should be restricted to low precision calculations. Our implementation of the asymptotic expansion exhibits fast convergence for large parameters, but the limitation on the achievable accuracy forces a switch to numerical integration depending on the desired level of precision.

\begin{table}[H]
	\centering
	\scalebox{0.85}{
\begin{tabular}{cccccccc}
\hline
$\Phi(z,s, a)$ & bits & mpmath & Mathematica & Euler-Maclaurin & Asymptotic & $K$ & peak time\\
\hline
	$z = 140$ & 64    & 0.0684  & 0.0154 & 0.0027 & 0.0022 & 9 & 6.3\%\\
	$s = 1/4$ & 333   & 0.7859  & 0.1219 & 2.2342 & 0.0134 & 59 & 10.9\%\\
	$a = 200$ & 1024  & 5.0361  & 0.7297 & - & 0.1931 & 254 & 10.5\%\\
	\hline
	           & 64   & 0.1238  & 0.0661 & - & 0.0017 & 6 & 4.6\%\\
	$z =10000$ & 333  & 1.5456  & 0.2078 & - & 0.0066 & 34 & 11.5\%\\
	$s = 10/4$ & 1024 & 9.9478  & 1.0406 & - & 0.0481 & 122 & 10.0\%\\
	$a = 2000$ & 3333 & 94.362  & 15.141 & - & 0.9498 & 493 & 8.4\%\\
	           & 10000 &1978.2  & 283.21 & - & 39.779 & 1952 & 6.0\%\\
	\hline
	\end{tabular}}
	\caption{Time (in seconds) to compute $\Phi(z,s,a)$ for large parameter $a$ and argument $z$. Comparison to Euler-Maclaurin at low precision. The rightmost column shows the percentage of the total time devoted to computation of $K$ peak numbers.}
	\label{table_large_z}
\end{table}

Finally, Table \ref{table_uae} compares the uniform asymptotic expansion (\ref{uae_lerch}) to the asymptotic expansion (\ref{asymp_large_z}). Results show that the former expansion should be the preferred choice at low-medium precision for sufficiently large parameters and argument, otherwise the previous methods generally show superior performance. Note that the computation of Eulerian polynomials accounts for the majority of the total time, consequently any improvement on this respect will directly reduce the reported timings.

\begin{table}[H]
	\centering
	\scalebox{0.85}{
\begin{tabular}{ccccccc}
\hline
$\Phi(z,s, a)$ & bits & mpmath & Mathematica & Asymptotic & Uniform & Eulerian time\\
\hline
	$z = -200.65$ & 64 & 0.0228* & 0.0469 & 0.0031 (16) & 0.0173 (20) & 97.0\%\\    
	$s = 100.25$ & 333 & 0.0149* & 0.1563 & 0.0412 (104) & 0.0592 (60) & 97.5\%\\
	$a = 501.5$ & 1024 & 1.0219** & 0.8438 & 0.7512 (421) & 0.5151 (229) & 98.8\%\\
	\hline
	$z = -20000$ & 64 & 0.0101* & 0.0312 & 0.0027 (15) & 0.0051 (9) & 93.2\%\\    
	$s = 100.25$ & 333 & 0.0168* & 0.1875 & 0.0362 (93) & 0.0436 (53) & 96.9\%\\
	$a = 501.5$ & 1024 & 1.7736 & 1.0313 & 0.5424 (365) & 0.2964 (196) & 98.3\%\\
	\hline
	\end{tabular}}
	\caption{Time (in seconds) to compute $\Phi(z,s,a)$ for large parameter $a$ and $s$, and argument $z$. Number of terms for each expansion within parentheses. For mpmath: (*) and (**) indicate no answer and inaccurate answer, respectively.}
	\label{table_uae}
\end{table}

\section{Discussion}
The algorithms presented in this work are an important step towards a complete arbitrary-precision implementation of the Lerch transcendent using asymptotically fast methods. A fundamental improvement to our implementation is to devise a more intelligent strategy to address cancellation issues for the Euler-Maclaurin formula, which should yield a significant reduction of the current overhead factor.

Further work is needed to develop an efficient multithreaded implementation of the asymptotic expansions. More importantly, it remains an open problem whether there is a fast memory-efficient algorithm for computing a large number of Eulerian and peak polynomials.

\end{document}